\documentclass[12pt]{article}

\usepackage{amsmath,amsthm,amssymb,amscd}
\theoremstyle{plain}
\newtheorem{thm}{Theorem}[section]
\newtheorem{prop}[thm]{Proposition}
\newtheorem{dfn}[thm]{Definition}
\newtheorem{lem}[thm]{Lemma}
\newtheorem{cor}[thm]{Corollary}
\theoremstyle{definition}
\newtheorem*{rem}{Remark}

\def\A{{\mathcal A}}
\def\D{{\mathcal D}}
\def\E{{\mathcal E}}
\def\EE{{\mathbf E}}
\def\F{{\mathcal F}}
\def\H{{\mathcal H}}
\def\area{{\rm area}}
\def\vol{{\rm vol}}
\def\supp{{\rm supp}}

\newcommand{\RR}{{\mathbb R}}
\newcommand{\CC}{{\mathbb C}}
\newcommand{\NN}{{\mathbb N}}

\begin{document}

\title{Integral geometric properties of non-compact harmonic spaces}

\author{Norbert Peyerimhoff \and Evangelia Samiou}

\date{\today}

\maketitle

\section{Introduction}

A complete Riemannian manifold $(X,g)$ of dimension $n+1$ is called harmonic, if the volume density function in normal coordinates around a point $x_0$ depends only on the distance from this point. Rank one symmetric spaces are harmonic, and Lichnerowicz conjectured that a simply connected harmonic space must be flat or rank one symmetric. For compact simply connected spaces this is true by a theorem of Szabo \cite{Sz-90}. However, certain 3-step solvmanifolds, constructed by Damek and Ricci \cite{DR-92}, provide examples of non-compact {\em non-symmetric} homogeneous harmonic spaces. Heber \cite{Heb-06} proved that there exist no other simply connected {\em homogeneous} harmonic spaces. Recently, Knieper \cite{Kn-12} showed for non-compact simply connected harmonic spaces that (i) having a purely exponential volume density function,  (ii) being Gromov-hyperbolic, and (iii) having a Anosov geodesic flow, are all equivalent conditions.

Non-compact harmonic spaces are manifolds without conjugate points. Moreover, they are Einstein and, therefore,
analyic by the Kazdan-De Turck theorem.

It was shown by Willmore \cite{Wi-50} that harmonic manifolds can also be characterized as those analytic spaces for which all harmonic functions $f$ satisfy the mean value property, namely, that the average of $f$ over any geodesic sphere equals the value of $f$ at its center. It is well-known that in a harmonic space every function satisfying the mean value property at all points for all radii must be harmonic. This is no longer true if a function only satisfies the mean value property at all points for a single radius $r_1 > 0$. A simple example of a non-harmonic function satisfying the mean value property for the radius $r_1 = 2 \pi$ is the cosine function on the real line $X = \RR$. We will show in this paper that in arbitrary non-compact harmonic spaces, the mean value property for two {\em generically chosen} radii implies harmonicity of the function, namely, this is true for {\em any fixed radius $r_1 > 0$ and any choice of $r_2 > 0$ avoiding a certain countable set} (see Theorem \ref{theomain3}, combined with Proposition \ref{prop:count}). In the above example the second radius $r_2$ has to avoid rational multiples of $r_1$.  Furthermore, we study the question whether a continuous function must vanish in case that all its integrals over certain spheres or balls are zero (see Theorem \ref{theomain1}).

The paper is organised as follows. In Section 2, we introduce basic
notions and define convolutions in harmonic spaces following ideas of
Szabo \cite{Sz-90}, and prove useful properties of them.

In Section 3, we derive fundamental results for the Abel transform and
the spherical Fourier transform, in particular, that the Abel
transform and its dual are topological isomorphisms, using finite
propagation speed of the wave equation and a D'Alembert type formula for the
Klein-Gordon equation.

In Section 4 we prove the above mentioned integral geometric results
for all non-compact harmonic manifolds. The arguments there are
analogous to our earlier paper \cite{PS-10}, where we studied two
radius results for Damek-Ricci spaces. This realises the proposed
research direction indicated in \cite[Section 10]{BZ-80}. A crucial
step is the reduction of the problem to a classical result of
L. Schwartz \cite{Schw-47} on mean periodic functions. For a modern
treatment of mean periodic functions in symmetric spaces, see
\cite{VV-09}.

Finally, in Section 5, we present some results related to the Cheeger
constant (Theorem \ref{thm:cheeg}) and to the heat kernel (Theorem
\ref{thm:heatk}) of non-compact harmonic manifolds.

\bigskip

{\bf Acknowledgements:} Both authors are grateful to the University of
Cyprus for the financial support. We also thank Professor
Yiorgos-Sokratis Smyrlis for useful conversations.

\section{Radial eigenfunctions and convolutions}

Henceforth, $(X,g)$ denotes a non-compact, complete, simply connected
harmonic space, $\theta(r)$ the density function of a geodesic sphere
of radius $r > 0$, $H \ge 0$ the mean curvature of all horospheres,
and $x_0 \in X$ a particular reference point. Let $r(x) :=
d(x_0,x)$. The closed ball of radius $r > 0$ around $x \in X$ is
denoted by $B_r(x) \subset X$. For the inner product, we use the
notation
$$ \langle f, g \rangle = \int_X f(x) g(x) dx. $$

Let $\D(X)$, resp., $\E(X)$ denote the vector space of smooth
functions on $X$, resp., smooth functions with compact support,
equipped with the topology of uniform convergence of all derivatives on compact sets, see \cite[Ch. II \S 2]{Hel-84} for instance.

\begin{dfn}
  For every $x \in X$, the {\em spherical projector} $\pi_x: \E(X) \to
  \E(X)$ is defined by
  $$ (\pi_x f)(y) := \frac{1}{\vol(S_r(x))} \int_{S_r(x)} f \quad
  \text{with $r=d(x,y)$,}
  $$
  where $S_r(x)$ denotes the geodesic sphere around $x$ with radius $r$.

  Let $\E_0(X,x) := \pi_x(\E(X))$ and $\D_0(X,x) :=
  \pi_x(\D(X))$. Functions in these spaces are called {\em radial
    functions} about $x$. We simply write $\pi, \E_0(X), \D_0(X)$ for
  $\pi_{x_0}, \E_0(X,x_0), \D_0(X,x_0)$. A radial function
  $f\in\E_0(X)$ is of the form
  \begin{equation} \label{eq:radfuncx0}
    f(x)=\tilde{f}(r(x))=\tilde{f}(d(x,x_0))
  \end{equation}
  with some even function $\tilde{f}\in\E_0(\RR)$. We often do not
  distinguish between $f$ and $\tilde{f}$ in our notation, i.e. we
  simply write $f(x)=f(r(x))$.
\end{dfn}

We now present basic properties of the spherical projector. The first
property below is obvious, and the second identity can be found, e.g.,
in \cite[Lemme 2]{Rou-03}.

\begin{lem} \label{prop:avprop}
  For $x \in X$, the operator $\pi_x$ has the following properties:
  \begin{eqnarray}
  \pi_x^2&=&\pi_x, \label{eq:idem}\\
  \langle \pi_x f, g \rangle&=&\langle f, \pi_x g \rangle.
  \end{eqnarray}
\end{lem}

The Laplacian $\Delta = {\rm div} \circ {\rm grad}$, applied to a
radial function $f \in \E_0(X)$, can be written as
\begin{equation} \label{eq:Lappolcoord}
(\Delta f)(r(x)) = f''(r(x)) + \frac{\theta'(r(x))}{\theta(r(x))} f'(r(x)).
\end{equation}
It is well known that $\frac{\theta'(r)}{\theta(r)}$ is the mean
curvature of a geodesic sphere $S_r(x)$, and that
$\frac{\theta'(r)}{\theta(r)}$ is a monotone decreasing function
converging to $H$ (see \cite[Cor. 2.1]{RaSha-03}).

Concerning eigenvalues, we follow the sign convention in
\cite{Cha-84}, and call $f \in C^\infty(X)$ an eigenfunction to the
eigenvalue $\mu \in \CC$ if $\Delta f + \mu f = 0$. However, we impose
no growth restriction on eigenfunctions, i.e., we do not require that
$f$ is an $L^2$-function.

We now prove uniqueness and existence of radial eigenfunctions of the
Laplacian. For positive eigenvalues this is shown in \cite{Sz-90}.

\begin{prop}\label{prop:efctLap} For each $\lambda\in\CC$ there is a
  unique smooth function $\varphi_\lambda\in\E_0(X)$ such that
  \begin{equation}\label{eq:evalLapl}
    \Delta\varphi_\lambda + \left( \lambda^2 + \frac{H^2}{4}\right)
    \varphi_\lambda = 0, \quad\text{and}\quad \varphi_\lambda(x_0)=1.
  \end{equation}
  We obviously have $\varphi_{-\lambda} = \varphi_\lambda$ and
  $\varphi_{iH/2}=\varphi_{-iH/2}=1$. Also $\varphi_\lambda(r)$ is
  holomorphic in $\lambda$.
\end{prop}

\begin{proof}
  We fix $\lambda\in\CC$ and abbreviate $L:= -\left( \lambda^2 +
    \frac{H^2}{4}\right)$. The eigenvalue equation \eqref{eq:evalLapl}
  translates to
  $$\varphi_\lambda''+\frac{\theta'}{\theta}\varphi_\lambda' =
  \frac{(\theta
    \varphi_\lambda')'}{\theta}=L\varphi_\lambda\quad\text{and}\quad
  \varphi_\lambda(0)=1.$$
  Integrating twice we get that this is equivalent to
  \begin{equation}\label{eq:radevalLaplint}
    \varphi_\lambda(r)=\varphi_\lambda(0) +
    L\int_0^r\frac{1}{\theta(r_2)}
    \int_0^{r_2}\theta(r_1)\varphi_\lambda(r_1)dr_1 dr_2=
    1+L\int_0^r q(r,r_1)\varphi_\lambda(r_1)dr_1,
  \end{equation}
  where
  $$q(r,r_1)=\int_{r_1}^r\frac{\theta(r_1)}{\theta(r_2)}dr_2.$$
  By \cite[Prop. 2.2]{RaSha-03} the function $\theta \ge 0$ increases,
  hence $0\leq q(r,r_1)\leq r-r_1$, and the Volterra integral equation
  of the second kind \eqref{eq:radevalLaplint} has a unique solution
  (see \cite[Thm 5]{Hoch-73}). In order to obtain a power series in
  $L$ for $\varphi_\lambda$, we use \eqref{eq:radevalLaplint}
  iteratively, starting with the constant function $1$, and obtain
  $$\varphi_\lambda(r)=1+\sum_{k=1}^\infty a_k(r) L^k$$
  with coefficients
  $$a_k(r)=\int_{r\geq r_1\geq r_2\geq\cdots r_k\geq 0} q(r,r_1)q(r_1,r_2) \cdots
  q(r_{k-1},r_k) dr_1\ldots dr_k.$$
  Since
  $$ 0 \le a_k(r) \le \int_{r\geq r_1\geq r_2\geq\cdots r_k\geq 0}
  (r-r_1)(r_1-r_2)    \cdots   (r_{k-1}-r_k)    dr_1\ldots   dr_k    =
  \frac{r^{2k}}{(2k)!}, $$
  the power series above converges for all $L\in\CC$.
\end{proof}


The following lemma can be found in \cite[Lemma 1.1]{Sz-90}:

\begin{lem} \label{lem:piDe}
  We have
  $$ \pi_x \circ \Delta = \Delta \circ \pi_x. $$
\end{lem}

An immediate consequence of Lemma \ref{lem:piDe} is the fact that
if $\Delta f + \mu f =0$ and $f\neq 0$ is radial, then $g:= \pi_x f \in
\E(X,x)$ is also an eigenfunction of $\Delta$ with eigenvalue
$\mu$.

The {\em displacement} of a radial function $f \in \E_0(X,x)$ at a
point $y \in X$ is denoted by $f_y \in \E_0(X,y)$ and defined by
$$ f_y(z) := \widetilde f(d(y,z)), $$
with $\widetilde f(d(x,y)) = f(y)$. We have $f_y(z) =
f_z(y)$.

\begin{lem}
  The displacement $(\varphi_\lambda)_x$ of an eigenfunction
  $\varphi_\lambda$, is again an eigenfunction to the same
  eigenvalue and
  $$
  \pi ((\varphi_\lambda)_x) = \varphi_\lambda(x) \varphi_\lambda.
  $$
\end{lem}

\begin{proof}
  Because of the representation \eqref{eq:Lappolcoord} of the
  Laplacian in polar coordinates, which is independent of the center,
  the displacement $(\varphi_\lambda)_x$ is also an eigenfunction to
  the eigenvalue $\mu = (\lambda^2 + H^2/4)$. From $\Delta \circ \pi =
  \pi \circ \Delta$ (Lemma \ref{lem:piDe}) we conclude that $\pi
  ((\varphi_\lambda)_x)$ is a radial eigenfunction about $x_0$ to the
  eigenvalue $\mu$ and, by uniqueness, a multiple of
  $\varphi_\lambda$. We have
  $$ (\pi ((\varphi_\lambda)_x))(x_0) = (\varphi_\lambda)_x(x_0) =
  (\varphi_\lambda)_{x_0}(x) = \varphi_\lambda(x), $$
  finishing the proof of the lemma.
\end{proof}

Next, we recall an important property of the convolution of radial
kernel functions in harmonic spaces (see \cite{Sz-90}). A smooth
function $F: X \times X \to \CC$ is called a {\em radial kernel
  function} if there is a function $\widetilde f: [0,\infty) \to \CC$
such that
$$
F(x,y) = \widetilde f(d(x,y)) \quad \text{for all $x,y \in X$.}
$$
$F$ is called {\em of compact support}, if there is a radius $R >
0$ such that $F(x,y) = 0$ for all $d(x,y) \ge R$.

\begin{prop} (see \cite[Prop. 2.1]{Sz-90}) \label{prop:szabo}
  Let $F,G: X \to X \to \CC$ be two radial kernel functions, one of
  them of compact support. Then the convolution
  $$ F*G(x,y) := \int_X F(x,z) G(z,y) dy $$
  is, again, a radial kernel function, i.e. $F*G(x,y)$ depends only
  on the distance $d(x,y)$.
\end{prop}

Every radial function $f \in \E_0(X)$ is in one-one correspondence with
a radial kernel function $F: C^\infty(X \times X)$ via
\begin{eqnarray*}
F(y,z) &=& \widetilde f(d(y,z)), \\
f(x) &=& F(x_0,x),
\end{eqnarray*}
where $\widetilde f$ was introduced in \eqref{eq:radfuncx0}. This
correspondence leads to a natural convolution $f * g$ of radial
functions $f,g \in \E_0(X)$, and even to an extension of this
notion if only one of the two functions is radial:

\begin{dfn}
  Let $f, g \in \E(X)$, one of them with compact support, and one of
  them radial about $x_0$. If $f$ is the radial function, the
  convolution $f * g \in \E(X)$ is defined as
  $$ f*g(y) := \langle f_y, g \rangle = \int_X f_y(z) g(z) dz\ .$$
  Similarly, if $g$ is the radial function, we define
  $$ f*g(y) := \langle f, g_y \rangle = \int_X f(z) g_y(z) dz\ .$$
\end{dfn}

\begin{rem}
  The convolution $f * g$ is well defined, since if both
  $f,g$ are radial, we have
  $$ \int_X f_y(z) g(z) dz = F*G(y,x_0) {\quad} \text{and}\;
  \int_X f(z) g_y(z) dz = F*G(x_0,y), $$
  where $F$ and $G$ are the radial kernel functions associated to
  $f,g$. By Proposition \ref{prop:szabo}, we have $F*G(y,x_0) =
  F*G(x_0,y)$. Moreover, the definition immediately implies
  commutativity of the convolution.
\end{rem}

The following two lemmata are further consequences of Proposition
\ref{prop:szabo}. We omit the proofs which are straightforward, once
the statements are reformulated in terms of radial kernel functions.

\begin{lem} \label{lem:convrad}
  Let $f, g \in \D_0(X)$. Then $f * g \in \D_0(X)$.
\end{lem}


\begin{lem} \label{lem:assconv23}
  Let $f \in \D(X)$ and $g, h \in \D_0(X)$. Then we have
  \begin{equation} \label{eq:convass}
  f*(g*h) = (f*g)*h.
  \end{equation}
\end{lem}


Now, let ${\D}'(X)$ and ${\E}'(X)$ be dual spaces of $\E(X)$ and
$\D(X)$, i.e. the space of distributions and the space of
distributions with compact support. For their topologies we refer,
again, to \cite[Ch. II \S 2]{Hel-84}. Let $\E_0'(X)$ and $\D_0'(X)$ be
the corresponding subspaces of radial distributions. The
spherical projector and the convolution are continuous and that
$\D_0(X)$ embeds canonically into $\E_0'(X)$ via $f \mapsto T_f$ where
$\langle T_f, g \rangle := \int_X f(x) g(x) dx$. For $T \in \E'(X)$
and $f \in \E_0(X)$, $T*f$ can be interpreted as a function in
$\E(X)$, i.e.,
$$ T*f(x) = \langle T, f_x \rangle. $$
Since $\D_0(X)$ lies dense in $\E_0'(X)$ (this follows by using a
Dirac sequence $\rho_\epsilon \in \D_0(X)$ and $T*\rho_\epsilon \to
T$), all above properties for functions carry over to distributions
(as, for instance, the fact that the convolution of two radial distributions
is radial, or the associativity of the convolution of radial
distributions.)

Since $\int_X g_y(z) f(z)dz = \int_X g(z) (\pi_y f)_{x_0}(z)$ for $g
\in \D_0(X)$ and $f \in \E(X)$, the convolution of $T\in\E_0'(X)$ and
$f\in\E(X)$ is given by
\begin{equation} \label{eq:Tconvf}
T*f(y) :=\left \langle T, (\pi_yf)_{x_0} \right\rangle  =
\left\langle T, z \mapsto \frac{1}{\vol(S_{r(z)}(y))}
\int_{S_{r(z)}(y)}f(x)dx \right\rangle.
\end{equation}
As an example, consider the distribution $T_r\in\E_0'(X)$, given by
$\langle T_r,f\rangle=\int_{S_r}f$. If $f\in\E(X)$ we obtain
\begin{equation} \label{eq:Trexamp}
T_r*f(y)= \left \langle T_r,
z \mapsto \frac{1}{\vol(S_{r(z)}(y))} \int_{S_{r(z)}(y)} f
\right \rangle = \int_{S_r(y)}f.
\end{equation}

The convolution of two radial distribution $S, T \in \E_0'(X)$
lies in $\E_0'(X)$ and can be written as follows: If $f \in \E_0(X)$, we
have
$$
\langle S*T, f \rangle = \langle S, x \mapsto \langle T,f_x
\rangle \rangle.
$$
For general $f \in \E(X)$, we have
$$
\langle S*T, f \rangle = \langle S*T, \pi f \rangle\ (= \langle S,
x \mapsto \langle T,(\pi f)_x \rangle \rangle).
$$

\begin{prop}\label{prop:pimodprop}
For $T\in\E_0'(X)$ and $f\in\E(X)$ we have
$$\pi (T*f)=T*(\pi f)\ .$$
\end{prop}

\begin{proof}
In view of \eqref{eq:Tconvf}, the claim means that
$$[\pi(T*f)](y)=\pi[u\mapsto\langle T, z\mapsto(\pi_u f)_{x_0}(z)\rangle](y)=\langle T, z\mapsto\pi[u\mapsto(\pi_u f)_{x_0}(z)](y)\rangle$$
is equal to
$$[T*(\pi f)](y)=\langle T, (\pi_y\pi f)_{x_0}\rangle$$
for all $y\in X$. To see this, we show that for all $y,z\in X$ we have
\begin{equation}\label{eq:pifmx34}\pi[u\mapsto(\pi_u f)_{x_0}(z)](y) = [(\pi_y\pi f)_{x_0}](z)\ .\end{equation}
We first show \eqref{eq:pifmx34} if $f$ is an eigenfunction of the Laplacian. So assume
$$f\in\E(X)\text{ with }\Delta f+\mu f = 0.$$
Then $\Delta\pi_u f+\mu\pi_u f=0$ and $\pi_u f\in\E(X,u)$ is
radial. Choose $\lambda \in \CC$ such that $\mu = \lambda^2 +
H^2/4$. By the uniqueness of the radial eigenfunctions we get
$$(\pi_u f)_{x_0}=f(u)((\varphi_\lambda)_u)_{x_0}=f(u)\varphi_\lambda.$$
Therefore
$$\pi[u\mapsto(\pi_u f)_{x_0}(z)](y)=\pi[u\mapsto f(u)\varphi_\lambda(z)](y)=\varphi_\lambda(z)[\pi f](y)=\varphi_\lambda(z)\varphi_\lambda(y)f(x_0)$$
and
$$[(\pi_y\pi f)_{x_0}](z)=[(\pi_y[f(x_0)\varphi_\lambda])_{x_0}](z)=f(x_0)\varphi_\lambda(y)\varphi_\lambda(z)\ .$$
In order to show \eqref{eq:pifmx34} for arbitrary functions $f$, note that for fixed $y,z\in X$ the values of both sides of \eqref{eq:pifmx34} depend on the restriction of $f$ to a compact subset $K\subset X$ only. Let $K'$ be a compact subset of $X$ with smooth boundary containing $K$ in its interior. Since $f$ is smooth, we find linear combinations of Dirichlet eigenfunctions of the Laplacian on $K'$ approximating $f$ uniformly on $K$. Since both sides of \eqref{eq:pifmx34} are continuous in $f$ with respect to uniform convergence, this establishes \eqref{eq:pifmx34} for all functions $f$.
\end{proof}

\section{Abel and spherical Fourier transformation}
\label{sec:abelfourier}

We first introduce some fundamental notions. Let $SX$ be the unit tangent
bundle of $(X,g)$. The {\em Busemann function} associated to a unit
tangent vector $v_0 \in S_{x_0}X$ is defined by
$$ b(x) = b_{v_0}(x) := \lim_{s \to \infty} ( d(c(s),x) - s), $$
where $c: \RR \to X$ is the geodesic with $c(0) = x_0$, $c'(0) =
v_0$. $\Delta b = H$ implies that $b$ is an analytic function. The
level sets of $b$ are smooth hypersurfaces and are called {\em
  horospheres}. They are denoted by
$$ \H_s := b^{-1}(s). $$
These horospheres foliate $X$ and we have $x_0 \in \H_0$.
We also need the smooth unit vector field
$$ N(x) = - {\rm grad} b(x), $$
orthogonal to the horospheres $\H_s$ and satisfying $N(x_0) = v_0$.
We choose an orientation of $\H_0$ and
orientations of $\H_s$ such that the diffeomorphisms
$$ \Psi_s: \H_0 \to \H_s, \quad \Psi_s(x) := \exp_x(-s N(x)), $$
are orientation preserving. Since $HN$ is the mean curvature vector $HN$ of the horospheres is the variation field of the volume functional we have
\begin{prop} \label{prop:relvol}
  We have
  $$ (\Psi_s)^* \omega_s = e^{s H} \omega_0. $$
\end{prop}
We combine the diffeomorphisms $\Psi_s: \H_0 \to \H_s$ to construct a
global diffeomorphism
$$ \Psi: \RR \times \H_0 \to X, \quad \Psi(s,x) := \Psi_s(x). $$
We have $D \Psi(\frac{\partial}{\partial s}) = -N$.

We choose an orientation on $\RR \times \H_0$ such that every oriented
base $v_1,\dots,v_n$ of $\H_0$ induces an oriented base
$\frac{\partial}{\partial s}, v_1, \dots, v_n$ on $\RR \times
\H_0$. This yields also an orientation on $X$ by requiring that $\Psi$
is orientation preserving. An immediate consequence of Proposition
\ref{prop:relvol} is

\begin{cor} \label{cor:horvolform} Let $\omega$ denote the volume form of the
harmonic space $(X,g)$. Then we have
  $$ \Psi^* \omega = e^{s H} ds \wedge \omega_0. $$
\end{cor}

Next, we fix a unit vector $v_0 \in S_{x_0}X$, and denote the
associated Busemann function $b_{v_0}$ by $b$, for simplicty . We first
consider the following important transform:

\begin{dfn}
  Let $j : \E_0(\RR) \to \E(X)$ be defined as
  $$ (jf)(x) = e^{-\frac{H}{2} b(x)} f(b(x)). $$
  The transformation $a: \E_0(\RR) \to \E_0(X)$ is then
  defined as
  $$ a = \pi \circ j. $$
  The {\em Abel transform} $\A: \E_0'(X) \to \E_0'(\RR)$ is
  defined as the dual of $a$, i.e., we have for all $T \in
  \E_0'(X)$ and $f \in \E_0(\RR)$:
  $$ \langle \A T, f \rangle_\RR = \langle T, a f \rangle_X, $$
  where $\langle f,g \rangle_\RR = \int_\RR f(s)g(s)ds$.
\end{dfn}

The functions $\psi_\lambda(s) = \frac{1}{2} (e^{i\lambda s} + e^{-i\lambda s}) = \cos(\lambda s)$ are the radial eigenfunctions of the Laplacian $\Delta f = f''$ on the real line.

\begin{lem} \label{lem:aeigen}
  We have
  $$ a \psi_\lambda = \varphi_\lambda. $$
\end{lem}

\begin{proof}
  We first observe that, under the diffeomorphism $\Psi: \RR \times \H_0 \to
  X$, the Laplacian has the form
  \begin{equation} \label{eq:Laphorocoord}
  \Delta = \frac{\partial^2}{\partial^2 s} + H \frac{\partial}{\partial
  s} +  A_s,
  \end{equation}
  where $A_s$ is a differential operator with derivatives tangent
  to $\H_0$. Consequently, the functions $f_\alpha = e^{-\alpha
  b}$ are eigenfunctions of $\Delta$ with
  $$ \Delta f_\alpha = - \alpha (H - \alpha) f_\alpha. $$
  Choosing $\alpha = \frac{H}{2} \pm i \lambda$, we obtain
  $$ \Delta f_\alpha = - \left( \lambda^2 + \frac{H^2}{4} \right)
  f_\alpha, $$
  and, by uniqueness, $\pi f_\alpha$ must be a multiple of
  $\varphi_\lambda$. Since $f_\alpha(x_0) = 1$, we conclude that
  $\varphi_\lambda = \pi f_\alpha$.
  Let $\alpha_\pm = \frac{H}{2} \pm i \lambda$. Then one easily
  checks that
  $$ j \psi_\lambda = \frac{1}{2} \left(e^{-\alpha_- b} + e^{-\alpha_+
  b}\right), $$
  and, consequently,
  $$ a \psi_\lambda = \frac{1}{2} \left(\pi f_{\alpha_-} + \pi
  f_{\alpha_+}\right) = \varphi_\lambda. $$
\end{proof}

Let $\psi_{\lambda,k}(s) = \frac{d^k}{d\lambda^k} \psi_\lambda(s)
= \frac{s^k}{2} (i^k e^{i\lambda s} + (-i)^k e^{-i\lambda s})$ and
$\varphi_{\lambda,k} = \frac{d^k}{d\lambda^k} \varphi_\lambda$.
Lemma \ref{lem:aeigen} implies that we also have
\begin{equation}\label{eq:apsiphi}
a \psi_{\lambda,k} = \varphi_{\lambda,k},
\end{equation}
for all $k \ge 1$.


\begin{prop}
  The Abel transform of a function $f \in \D_0(X) \subset
  \E_0'(X)$ is $\A f\in\D_0(\RR)\subset\E_0'(\RR)$ given by
  \begin{equation} \label{eq:Af}
  (\A f)(s) = e^{-\frac{H}{2} s} \int_{\H_s} f(z) \omega_s(z) =
  e^{\frac{H}{2} s} \int_{\H_0} f(\Psi_s(z)) \omega_0(z).
  \end{equation}
\end{prop}

\begin{proof}
  Let $f \in \D_0(X)$ and
  $$ g(s) = e^{-\frac{H}{2} s} \int_{\H_s} f(z) \omega_s(z). $$
  Since $f$ has compact support, there is $T > 0$ such that $\H_s
  \cap \supp f = \emptyset$ for all $\vert s \vert \ge T$, i.e.,
  $g$ has also compact support. Moreover, by Proposition
  \ref{prop:relvol}, we obtain
  \begin{eqnarray*}
  g(s) &=& e^{-\frac{H}{2} s} \int_{\Psi_s(\H_0)} f(z)
  \omega_s(z) = e^{-\frac{H}{2} s} \int_{\H_0} f(\Psi_s(z))
  (\Psi_s^* \omega_s)(z)\\
  &=& e^{\frac{H}{2} s} \int_{\H_0} f(\Psi_s(z)) \omega_0(z).
  \end{eqnarray*}
  Next, we show $\langle g, h \rangle = \langle f, a h \rangle $
  for all $h \in \E_0(\RR)$:
  \begin{eqnarray*}
  \langle g,h \rangle &=& \int_{-\infty}^\infty g(s) h(s) ds =
  \int_{-\infty}^\infty e^{- \frac{H}{2} s} h(s) \int_{\H_s} f(z)
  \omega_s(z) ds \\
  &=& \int_{-\infty}^\infty \int_{\H_s} f(z) e^{-\frac{H}{2} b(z)}
  h(b(z)) \omega_s(z) ds = \int_X f(z) e^{-\frac{H}{2} b(z)}
  h(b(z)) dz \\
  &=& \langle f, j h \rangle = \langle \pi f, j h \rangle =
  \langle f, a h \rangle.
  \end{eqnarray*}
  Finally, we show that $g$ is an even function: Using $\langle g,h
  \rangle = \langle f,ah \rangle$ and Lemma \ref{lem:aeigen}, we derive
  $$ \int_{-\infty}^\infty g(s) e^{\pm i \lambda s} ds = \langle
  f, \varphi_\lambda \rangle, $$
  which implies that
  $$
  \int_{-\infty}^\infty e^{i \lambda s} (g(s) - g(-s)) ds = 0,
  $$
  for all $\lambda \in \CC$. This yields $g(s) = g(-s)$, i.e., $g$
  is an even function. This finishes the proof of the proposition.
\end{proof}

\begin{lem}\label{dualALaplace} For $f\in\E'_0(X)$, we have
$$\A(\Delta f) = \left(\frac{d^2}{ds^2}-\frac{H^2}{4}\right) \A f\ .$$
\end{lem}

\begin{proof} Let $u\in\E_0(\RR)$. Since $ju\in\E(X)$
  is constant on the horospheres, from \eqref{eq:Laphorocoord}, we
  compute
  $$
  [\Delta j u](x) = \left.\left(\frac{\partial^2}{\partial
        s^2}+H\frac{\partial}{\partial s}\right) e^{-s H/2}
    u(s)\right|_{s=b(x)} =\left[j\left(\frac{d^2}{d
        s^2}-\frac{H^2}{4}\right)u\right](x),
  $$
  hence
  \begin{eqnarray*}
    \langle\A(\Delta f), u\rangle&=&\langle\Delta f, au\rangle = \langle f,\Delta \pi j u\rangle=\langle f,\pi\Delta j u\rangle=\langle f,\pi j \left(\frac{d^2}{d s^2}-\frac{H^2}{4}\right) u\rangle\\
    &=&\langle f, a \left(\frac{d^2}{d s^2}-\frac{H^2}{4}\right) u\rangle=\langle \A f,\left(\frac{d^2}{d s^2}-\frac{H^2}{4}\right) u\rangle=\langle \left(\frac{d^2}{d s^2}-\frac{H^2}{4}\right)\A f, u\rangle.
  \end{eqnarray*}
\end{proof}

The next result will be used in the proof of Theorem \ref{thm:abij} below, namely to establish the local injectivity of $a: \E_0(\RR) \to \E_0(X)$.

\begin{lem}\label{lem:horowavesol}
Let $g\in\D(\RR)$. Then the Klein-Gordon equation
\begin{eqnarray}
\label{eq:hwe}\frac{\partial^2}{\partial t^2}v(t,s)&=&\frac{\partial^2}{\partial s^2}v(t,s)-\frac{H^2}{4}v(t,s)\\
\label{eq:hewic}v(0,s)&=&g(s)\text{ and }\frac{\partial}{\partial t}v(0,s)=0
\end{eqnarray}
has a solution of the form
\begin{equation} \label{eq:horosphwavesolform}
  v(t,s)=\frac{g(s-t)+g(s+t)}{2}+\int_{s-t}^{s+t}
  W(t,s-s')g(s')ds',
\end{equation}
with $W\in\E(\RR^2)$. This solution is unique in the sense that if
$\tilde v$ is another solution of \eqref{eq:hwe} and \eqref{eq:hewic} and
so that for all $t$ the function $\tilde v_t\colon s\mapsto\tilde v(t,s)$
has compact support, then $\tilde v=v$.
\end{lem}

\begin{proof}
The function $W$ is explicitly given by
\begin{equation}\label{eq:kernel-0}
W(t,s) = t\sum_{k=0}^\infty\left(\frac{-H^2}{16}\right)^{k+1}\frac{(t^2-s^2)^k}{k!(k+1)!},
\end{equation}
but we will only need that $W$ is smooth. A straightforward computation shows that \eqref{eq:horosphwavesolform} actually solves \eqref{eq:hwe} and \eqref{eq:hewic}. The function $W$ is even in the second argument $s$, and solves the equations
$$W(t,t)=-\frac{H^2}{16}t,\  W_{tt}=W_{ss}-\frac{H^2}{4}W. $$
Uniqueness of the solution follows from conservation of energy as, for instance, in [Taylor, pp145]. To see this, assume that $\tilde v$ is another solution of \eqref{eq:hwe} and \eqref{eq:hewic}. Then the difference $\omega=v-\tilde v$ solves \eqref{eq:hwe} and \eqref{eq:hewic} with $g$ replaced by $0$. We now look at the energy
$$E_\omega(t):=\int_{-\infty}^\infty\omega_s(t,s)^2+\omega_t(t,s)^2+\frac{H^2}{4}\omega(t,s)^2ds,$$
and compute, integrating by parts,
\begin{eqnarray*}
\frac{d}{dt}E_\omega(t)&=&2\int_{-\infty}^\infty\omega_s\omega_{st}+\omega_{tt}\omega_t+\frac{H^2}{4}\omega_t\omega ds=2\int_{-\infty}^\infty-\omega_{ss}\omega_t+\omega_{tt}\omega_t+\frac{H^2}{4}\omega\omega_t ds\\
&=&2\int_{-\infty}^\infty\left(-\omega_{ss}+\omega_{tt}+\frac{H^2}{4}\omega\right)\omega_t ds=0,
\end{eqnarray*}
because $\omega$ satisfies \eqref{eq:hwe}. Since $E_\omega(0)=0$ we have $E_\omega(t)=0$ for all $t$ which forces $\omega=0$.
\end{proof}

\begin{thm}\label{thm:abij}
  The maps $a: \E_0(\RR) \to \E_0(X)$ and $\A = a'$ are topological
  isomorphisms.
\end{thm}

\begin{proof}
We first show local injectivity of $a$. For all $R\geq 0$ the map $a$ induces a well defined map
$$a^R\colon\E_0([-R,R])\to\E_0(B_R(x_0))\ ,$$
i.e. for $x\in X$ the value of $au(x)$ depends only on the restriction of $u$ to $[-r(x),r(x)]$. Local injectivity now is the fact that for all $R\geq 0$, the maps $a^R$ are injective. Thus for $u\in\E_0(\RR)$ we have
$$ a u|_{B_R(x_0)} = 0 \quad\Longrightarrow\quad u|_{[-R,R]} = 0\ .$$
The proof is based on the fact that $\A\colon\E_0'(X)\to\E_0'(\RR)$
transforms the fundamental solution of the radial wave equation to the
fundamental solution of the Klein-Gordon equation. Since we need $\A$
on compactly supported distributions, finite propagation speed of the
solution of the wave equation is essential here.

For all $\epsilon>0$ we choose a function $q_\epsilon\in\E_0(X)$ so that
$$q_\epsilon(x)=0\text{ if }r(x)>\epsilon\text{ and }\int_Xq_\epsilon = 1\ .$$
Let $w\in\E_0(\RR\times X)$ be the solution of the wave equation starting with $q_\epsilon$, i.e.
\begin{eqnarray*}
\frac{\partial^2}{\partial t^2}w(t,r(x))&=&\Delta_x w_t(r(x))=\frac{\partial^2}{\partial s^2}w(t,r(x))+\frac{\theta'}{\theta}(r(x))\frac{\partial}{\partial s}w(t,r(x)),\\
w(0,s)&=&q_\epsilon(s),\\
\frac{\partial}{\partial s}w(0,s)&=&0,
\end{eqnarray*}
where $w_t(s) := w(t,s)$. By the finite propagation speed of the wave equation, the support of $w_t$ is compact (in fact contained in $B_{\epsilon+t}(0)$).

Let $v\in\E(\RR^2)$ be so that $v_t:=\A w_t$, i.e.
$$v(t,s)=e^{-Hs/2}\int_{\H_s}w(t,r(x))d\omega_s.$$
Then by Lemma \ref{dualALaplace}
$$\frac{\partial^2}{\partial t^2}v_t=\frac{\partial^2}{\partial t^2}\A w_t=\A\frac{\partial^2}{\partial t^2}w_t=\A\Delta w_t=\left(\frac{\partial^2}{\partial s^2}-\frac{H^2}{4}\right)v_t.$$

It follows that $v$ solves the Klein-Gordon equation \eqref{eq:hwe} with initial conditions (in place of \eqref{eq:hewic})
$$v_0=g_\epsilon:=\A q_\epsilon\text{ and }\frac{\partial}{\partial t}v(0,s)=0.$$
If $|s|>\epsilon$ then $v(0,s)=g_\epsilon(s)=0$. Also
$$\int_{-\infty}^\infty e^{Hs/2} g_\epsilon(s)ds=\int_{-\infty}^\infty\int_{\H_s}q_\epsilon(r(x))d\omega_s ds=\int_Xq_\epsilon=1.$$

Now, to prove local injectivity of $a$, let $R\geq 0$ and $u\in\E_0(\RR)$ with $a u|_{B_R(x_0)} = 0$. For all $t\in[-(R-\epsilon),R-\epsilon]$ we then have from \eqref{eq:horosphwavesolform}
\begin{eqnarray*}
0&=&\langle w_t, a u\rangle=\langle\A w_t, u\rangle=\langle v_t, u\rangle=\\
&=&\int_{-\infty}^\infty\frac{g_\epsilon(s-t)+g_\epsilon(s+t)}{2}u(s)ds+\int_{-\infty}^\infty\int_{s-t}^{s+t} W(t,s-s')g_\epsilon(s')ds'u(s)ds.
\end{eqnarray*}
Since $W,u$ and $g_\epsilon$ are smooth, we can take the limit $\epsilon\to 0$ here to get the identity
$$\frac{u(t)+u(-t)}{2}=-\int_{-t}^{t} W(t,s)u(s)ds.$$
Since $u$ and $W$ are even (in $s$), we can write this as a fixed point equation,
$$u(t)=-2\int_0^t W(t,s)u(s)ds.$$
Since $\pi f(x_0)=f(x_0)$ for all $f\in\E(X)$ we have $u(0)=0$. Let
$$M:=\max_{t\in[0,R],s\in[0,t]}\left|W(t,s)\right|.$$
Hence, if $u|_{[0,R]}\neq 0$ there is some $T\in[0,R]$ with the
following properties: (i) $u(T) \neq 0$ and (ii) $T<\frac{1}{4M}$ or
$u|_{[0,T-\frac{1}{4M}]}=0$. Now for all $t\in[0,T]$ we estimate
$$|u(t)|\leq 2\int_0^{t}|W(t,s)|\,|u(s)|ds\leq 2\int_{\max\{0,t-\frac{1}{4M}\}}^t |W(t,s)|\,|u(s)|ds\leq\frac{1}{2}\sup_{t'\in[0,T]}|u(t')|,$$
contradicting our assumption about $T$ and $u(T)\neq 0$.

Now we prove surjectivity of $a$: Let $f\in\E_0(X)$ and $R>0$ be
fixed. Choose a function $\phi\in \D_0(X)$ with $\phi = 1$ on
$B_R(x_0)$, $0 \le \phi \le 1$ on $X$, and $\phi = 0$ on $X \backslash
B_{R+1}(x_0)$.
We will first show that $\phi f|_{B_R(x_0)}$ is in $a(\E_0(\RR))|_{B_R(x_0)}$.

Let $\varphi_k$ be an orthonormal basis of Dirichlet eigenfunctions of
the Laplacian on $B_{R+1}(x_0)\subset X$ with corresponding
eigenvalues $0 \le \mu_k \nearrow \infty$. We have $\pi \varphi_k = \varphi_k(x_0)
\varphi_{\lambda_k}$ with $\lambda_k \in \CC$ such that $\mu_k = \lambda_k^2 + \frac{H^2}{4}$. Let
$$
\phi f = \sum_{k=0}^\infty a_k \varphi_k
$$
be the Fourier expansion of $\phi f$. Therefore,
$$
\phi f = \pi(\phi f) = \sum_{k=0}^\infty a_k \varphi_k(x_0)
\varphi_{\lambda_k}.
$$
Our first goal is to show that the series $(\sum_{k=0}^N a_k
\varphi_k(x_0)\cos(\lambda_k s))_{N \in \NN}$ converges uniformly with all its derivatives to a smooth function $g_R \in \E_0([-R,R])$. Then we have $a^R(g_R) = \phi f
\vert_{B_R(x_0)}$, by the continuity of $a$. We prove this by
showing that
\begin{equation} \label{eq:goal}
\sum_{k=0}^\infty |a_k|\; |\varphi_k(x_0)|\; |\lambda_k|^m < \infty \quad
\text{for all $m \in \NN$}.
\end{equation}

By the Sobolev imbedding theorem, there is a constant $C_0$ such that
for all $u \in C(B_{R+1}(x_0))$ we have
$$\|u\|_\infty \leq C_0 \left(\|u\|_2 + \| \Delta^{n+1} u\|_2\right)\ ,$$
where $\Vert \cdot \Vert_\infty$ denotes the supremum norm and $\Vert \cdot \Vert_2$ denotes the $L^2$-norm, and $n+1$ is the dimension of $X$. This implies that
$$
\|\varphi_k(x_0)| \le \|\varphi_k\|_\infty \le C_0(1+ \mu_k^{n+1}).
$$
By Weyl's law, the eigenvalues $\mu_k$ grow with an exponent $2/(n+1)$,
which implies that there is an $k_0 \in \NN$ and a $C_1 > 0$ such that,
for all $k \ge k_0$:
\begin{equation} \label{eq:hilf1}
|\varphi_k(x_0)| \le \|\varphi_k\|_\infty \le C_1 k^2.
\end{equation}

The Fourier expansion of $\Delta^{\nu} (\phi f)$ is given by
$$ \Delta^{\nu} (\phi f) = \sum_{k=0}^\infty a_k \mu_k^{\nu} \varphi_k, $$
and, since $\Delta^{\nu} (\phi f) \in L^2(B_{R+1}(x_0))$, we have
$$ \sum_{k=1}^\infty |a_k|^2 |\mu_k|^{2\nu} < \infty, $$
which implies that
\begin{equation} \label{eq:hilf2}
\sum_{k=1}^\infty |a_k|^2 |\lambda_k|^{4 \nu} < \infty,
\end{equation}
for every $\nu \in \NN$. We have
\begin{eqnarray*}
\sum_{k=0}^\infty |a_k|\; |\varphi_k(x_0)|\; |\lambda_k|^m &\le& C_1 \sum_{k=0}^\infty
  |a_k|\; k^2 \; |\lambda_k|^m \qquad \text{by \eqref{eq:hilf1}}, \\
&=& C_1 \sum_{k=0}^\infty \left(|a_k|\; k^2 \; |\lambda_k|^{m+l}\right)\; |\lambda_k|^{-l} \\
&\le & C_1 \left( \sum_{k=0}^\infty |a_k|^2 k^4\; |\lambda_k|^{2m+2l} \right)^{1/2}
\left(\sum_{k=0}^\infty |\lambda_k|^{-2l} \right)^{1/2} \\
&\le & C_2 \left( \sum_{k=0}^\infty |a_k|^2\; |\lambda_k|^{2m+2l+4(n+1)} \right)^{1/2}
\left( \sum_{k=0}^\infty |\lambda_k|^{-2l} \right)^{1/2} \qquad \text{by Weyl}.
\end{eqnarray*}
The required finiteness \eqref{eq:goal} follows now from
\eqref{eq:hilf2} and Weyl's law, for the choice $l = n+1$.

Hence $g_R$ defines a smooth function with $a^Rg_R=f$ on $B_R(x_0)$. Now
for a given $f\in\E_0(X)$ and each $N\in\NN$, construct
$g_N\in\E_0([-N-1,N+1])$ as above. We will have $a^N g_N =
f|_{B_N(x_0)}$. By local injectivity of $a$, $g_{N+1}|_{[-N,N]}=g_N$,
and the functions $g_N$ patch together to define a function
$g\in\E_0(\RR)$ with $ag=f$.

This shows that $a$ is a bijective linear continuous map. By the open
mapping theorem \cite[Theorem 17.1]{Tr-67}, $a$ is a topological isomorphism. Using the corollary of Proposition 19.5 in \cite{Tr-67}, we conclude that
its dual $\A: \E_0'(X) \to \E_0'(\RR)$ is also a topological
isomorphism.
\end{proof}

\begin{dfn}
  Let $f \in \D_0(X)$. The {\em spherical Fourier transformation} $\F f:
  \CC \to \CC$ is defined as
  $$ \F f(\lambda) := \langle f, \varphi_\lambda \rangle = \int_X
  f(x) \varphi_\lambda(x) dx. $$
  Analogously, the {\em spherical Fourier transformation} $\F T:
  \CC \to \CC$ of a radial distribution $T \in \E_0'(X)$ is
  defined as
  $$ \F T(\lambda) := \langle T, \varphi_\lambda \rangle. $$

\end{dfn}

Next, we will see that there is a close relationship between the Abel
transform $\A$ and the spherical Fourier transform $\F$. By the
classical Paley-Wiener theorem for distributions (see, e.g., \cite[p.
211]{Do-69} and \cite[Thm. 5.19]{Ehr-70}), the Euclidean Fourier
transform $\E_0'(\RR) \ni S \mapsto \hat S$, with $\hat S(\lambda) :=
\langle S, \psi_\lambda \rangle$ is a topological isomorphism
$\E_0'(\RR) \to \EE_0'$, where $\EE_0'$ is the space of all even
entire functions $f: \CC \to \CC$ of exponential type which are
polynomially bounded on $\RR$, endowed with a suitable topology.

\begin{prop} \label{prop:ATFT}
  We have
  $$ \widehat {\A T} = \F T. $$
\end{prop}

\begin{proof} We have for $\lambda \in \CC$,
  $$ \widehat {\A T}(\lambda) = \langle \A T, \psi_\lambda \rangle
  = \langle T, a \psi_\lambda \rangle = \langle T, \varphi_\lambda
  \rangle = \F T(\lambda). $$
\end{proof}

\begin{rem}
The explicit description of $\A f$ in \eqref{eq:Af} gives the
impression that the Abel transform depends on the choice of the
point $x_0 \in X$ as well as the unit vector $v_0 \in S_{x_0}X$,
since all horospheres $\H_s \subset X$ depend on $v_0$. But, in
fact, $\A f$ only depends on the choice of $x_0$. This follows
from the fact that the Abel transform can be recovered from the
Euclidean Fourier inverse of the spherical Fourier transform.
Obviously, the spherical Fourier transform does only depend on the
choice of $x_0$.
\end{rem}

\begin{prop} \label{prop:AST}
  For $S, T \in \E_0'(X)$ we have
  $$ \A (S * T) = \A S * \A T. $$
\end{prop}

\begin{proof}
  Note that $\A S, \A T, \A (S*T) \in \E_0'(\RR)$. It was shown in
  \cite[formula (30.1)]{Tr-67} that we have, for these
  distributions on the real line,
  $$ \widehat{\A S} \cdot \widehat{\A T} = \widehat{\A S * \A T}. $$
  We show now that $\widehat{\A (S*T)} = \widehat {\A S} \cdot
  \widehat {\A T}$:
  \begin{eqnarray*}
  \widehat{\A (S*T)}(\lambda) &=& \langle \A (S*T), \psi_\lambda
  \rangle = \langle S*T, \varphi_\lambda \rangle \\
  &=& \langle S, x \mapsto \langle T, (\varphi_\lambda)_x \rangle
  \rangle = \langle S, x \mapsto \langle T, \pi( (\varphi_\lambda)_x ) \rangle
  \rangle \\
  &=& \langle S, x \mapsto \langle T, \varphi_\lambda(x)
  \varphi_\lambda \rangle \rangle = \langle S,
  \langle T, \varphi_\lambda \varphi_\lambda \rangle \rangle \\
  &=& \langle S,a \psi_\lambda \rangle \langle T,a \psi_\lambda
  \rangle = \langle \A S, \psi_\lambda \rangle \cdot \langle \A T,
  \psi_\lambda \rangle = \widehat{\A S}(\lambda) \cdot \widehat{\A T}(\lambda).
  \end{eqnarray*}
  Putting both results together, we conclude that
  $$ \widehat{\A (S*T)} = \widehat{\A S * \A T}. $$
  Since the Euclidean Fourier transform $\E_0'(\RR) \to \EE_0'$,
  $T \mapsto \hat T$, is a topological isomorphism, we finally
  obtain
  $$ \A (S*T) = \A S * \A T, $$
  finishing the proof.
\end{proof}

An immediate consequence of Paley-Wiener for radial distributions
in Euclidean space, Theorem \ref{thm:abij} and Propositions
\ref{prop:ATFT} and \ref{prop:AST} is the following:

\begin{thm}[Paley-Wiener for radial distributions]
  The spherical Fourier transform
  $$ \F T(\lambda) = \langle T, \varphi_\lambda \rangle $$
  defines a topological isomorphism
  $$ \F: \E_0'(X) \to \EE_0'. $$
  Furthermore, for radial distributions $S,T \in \E_0'(X)$, we have
  $$ \F (S*T) = \F S \cdot \F T. $$
\end{thm}

\begin{prop} \label{prop:commdiag}
The following diagramm commutes:
\begin{equation}\begin{CD}
       \E_0'(X) \times \E_0(X) @>*_X>> \E_0(X)\\
       @VV{\A \times a^{-1}}V @VV{a^{-1}}V \\
       \E_0'(\RR) \times \E_0(\RR) @>*_\RR>> \E_0(\RR)
   \end{CD}
\end{equation}
\end{prop}

\begin{proof}
  Since $a: \E_0(\RR) \to \E_0(X)$ and $\A: \E_0'(X) \to
  \E_0'(\RR)$ are topological isomorphisms, in view of
  Proposition \ref{prop:AST}, it only remains  to show that $a( \A T *_\RR f) = T
  *_X a f$: For $g,h \in \D_0(X)$ we obtain
  \begin{eqnarray*}
    \langle g * a f, h \rangle &=& ((g * a f) * h)(x_0) = ((a f * g)
    * h)(x_0) = (a f * (g * h))(x_0) \\
    &=& \langle a f, g*h \rangle = \langle f, \A g * \A h \rangle
    \\
    &=& (f * (\A g * \A h))(x_0) = ((\A g * f) * \A h)(x_0) \\
    &=& \langle \A g * f, \A h \rangle = \langle a(\A g * f), h
    \rangle.
  \end{eqnarray*}
  Since $\D_0(X)$ is dense in $\E_0'(X)$ and $a, \A$ are
  continuous, we conclude the required identity.
\end{proof}

\section{Spectral analysis/synthesis and two radius theorems}\label{sec:tworadthms}

In this section, we discuss the proofs the integral geometric results mentioned in the Introduction. Since the proofs of very similar to the ones given in \cite{PS-10} for Damek-Ricci spaces, we give the ideas and outlines of the proofs, and refer to that paper for more details.

The following proposition is a consequence of the holomorphicity of the map $\lambda \mapsto \varphi_\lambda(r)$, and guarantees that the integral geometric results hold for two {\em generic radii}, as explained in the Introduction.

\begin{prop} \label{prop:count}
For each $r_1>0$ there is an at most countable set of $r_2>0$ such that there exists $\lambda\in\CC$ with $\varphi_\lambda(r_1)=\varphi_\lambda(r_2)=0$.
\end{prop}

\begin{proof}
Let $r_1>0$. By Proposition \ref{prop:efctLap}, the set $S_{r_1}=\{\lambda\in\CC\mid\varphi_\lambda(r_1)=0\}$ is at most countable. The zero set $\varphi_\lambda^{-1}(0)$ of $\varphi_\lambda$, is also at most countable because $\varphi_\lambda$ satisfies the differential equation $\varphi_\lambda''=-\left(\lambda^2+\frac{H^2}{4}\right)\varphi_\lambda-\frac{\theta'}{\theta} \varphi_\lambda$, the solution of which is determined by the values of $\varphi_\lambda(r_0)$ and $\varphi_\lambda'(r_0)$ for any $r_0>0$. Thus $\varphi_\lambda$ can not have a limit point of zeros. It follows that the set
$$\{r_2>0\mid\exists\lambda\in\CC : \varphi_\lambda(r_1)=\varphi_\lambda(r_2)=0\}=\bigcup_{\lambda\in S_{r_1}}\varphi_\lambda^{-1}(0)$$
is an at most countable union of at most countable sets, hence itself at most countable.
\end{proof}
The same reasoning applies to the function $\varphi_\lambda-1$ (we must exclude here $\lambda = \pm i \frac{H}{2}$, since then $\varphi_\lambda \equiv 1$), and the function $\Phi_\lambda$ given by
\begin{equation} \label{eq:Phil}
\Phi_\lambda(r)=\int_0^r\theta(\rho)\varphi_\lambda(\rho)d\rho\,
\end{equation}
where $\theta$ denotes the volume density function.

Analogously to \cite{PS-10}, we have {\em spectral analysis} and {\em spectral synthesis} in $\E_0(X)$. Let us briefly explain this.

A {\em variety} $V\subset\E_0(X)$ is a closed subspace satisfying
$\E_0'(X) * V \subset V$, which is proper ($V \neq \E_0(X)$) and contains a non-zero function.

It follows from Propositions \ref{prop:commdiag} and \ref{prop:AST},
together with the isomorphism Theorem \ref{thm:abij}, that the transformation $a$ maps
varieties of $\E_0(\RR)$ to varieties of $\E_0(X)$, and that every
variety in $\E_0(X)$ is of the form $a(W)$ with a variety
$W\subset\E_0(\RR)$. By Schwartz's Theorem on varieties in $\E_0(\RR)$ (see Theorem 2.4 in \cite{PS-10}), we have
$$W=\overline{\mathrm{span}\{\psi_{\lambda,k}\in W\}},$$
that is, each variety is the closure of the span of its spectrum, where
the spectrum of a variety is the set of all those $\psi_{\lambda,k}$
contained in $V$. Because of \eqref{eq:apsiphi} this carries over to
$X$, i.e. we have for any variety $V\subset\E_0(X)$ that
$$V=\overline{\mathrm{span}\{\phi_{\lambda,k}\in V\}}.$$
This property is called {\em spectral synthesis}. For varieties $V
\subset \E_0(X)$, it was shown in \cite[Lemma 4.2]{PS-10} that
$\varphi_{\lambda,k} \in V$ implies $\varphi_{\lambda,l} \in V$, for
all $0 \le l \le k$. There $X$ was a Damek-Ricci space, but the
arguments carry over verbatim for general harmonic
manifolds. Therefore, spectral synthesis implies that every
variety $V \subset \E_0(X)$ contains a radial eigenfunction $\varphi_\lambda$. This
latter property is called {\em spectral analysis}.

As an immediate application, we can prove the analogues of the Two-Radius Theorems in \cite{PS-10} for general non-compact harmonic spaces.

\begin{thm} \label{theomain1}
Let $(X,g)$ be a simply connected, non-compact harmonic manifold. Then we have the following facts.
\begin{enumerate}
\item Let $r_1, r_2 > 0$ be such that the equations
  \[ \varphi_\lambda(r_j)=0,\, j=1,2, \]
  have no common solution $\lambda \in \CC$. Suppose $f \in C(X)$ and
  \[ \int_{S_r(x)} f = 0 \]
  for $r=r_1,r_2$ and all $x \in X$. Then $f=0$.
\item For $\lambda\in\CC$, let $\Phi_\lambda$ be given by \eqref{eq:Phil}.
  Let $r_1, r_2 > 0$
  be such that the equations
  \[ \Phi_\lambda(r_j)=0,\, j=1,2, \]
  have no common solution $\lambda\in\CC$. Suppose $f\in C(X)$  and
  \[ \int_{B_r(x)} f = 0 \]
  for $r=r_1,r_2$ and all $x\in X$. Then $f=0$.
\end{enumerate}
\end{thm}

\begin{proof} Let us start with the proof of the first assertion.
  The proof proceeds by contradiction, i.e., let $g \in C(X)$ be a {\em non-zero} function, satisfying $\int_{S_{r_i}(x)} g = 0$ for all $x \in X$ and $i=1,2$. We choose our reference point $x_0 \in X$ such that $g(x_0) \neq 0$. Consider the distributions $T_r\in\E'_0(X)$, given by
 $T_rf=\int_{S_r(x_0)}f$. First recall \eqref{eq:Trexamp}, namely,
 $\int_{S_r(x)}f=(T_r*f)(x)$. This implies that we have $T_{r_1}*g = T_{r_2}*g = 0$.  Using the extension of Proposition \ref{prop:pimodprop} to continuous functions $f \in C(X)$, we conclude that $g_0 := \pi g$ also satisfies $T_{r_1}*g_0 =  T_{r_2}*g_0 = 0$. Without loss of generality, we can assume that $g_0$ is a smooth radial function, since every continuous radial function $g_0$ can be approximated, uniformly on compact sets, by functions $g_0^\epsilon := g_0 * \rho_\epsilon \in \E_0(X)$ (via a Dirac sequence $\rho_\epsilon \in \D_0(X)$), such that we still have $T_{r_1}*g_0^\epsilon = T_{r_2}*g_0^\epsilon = 0$. Therefore, we can now assume that $g_0 \in \E_0(X)$. Then
  $$ V = \left\{ f \in \E_0(X) \mid
  \int_{S_{r_1}(x)} f = 0 = \int_{S_{r_2}(x)} f \text{ for all } x \in X\right\}
  = \{ f\in\E_0(X) \mid T_{r_1}*f = 0 = T_{r_2}*f\}$$
  contains $g_0$ and is a variety in $\E_0(X)$, since for all $T \in \E_0'(X)$
  and all $f \in V$:
  $$ T_{r_i}*(T*f) = T*(T_{r_i}*f) = 0, \qquad \text{with\, $i=1,2$}, $$
  by commutativity and associativity of the convolution for radial
  distributions. By spectral analysis, $V$ must contain a
  $\varphi_\lambda$. But $T_r(\varphi_\lambda)=\vol(S_r(x_0))\varphi_\lambda(r)$,
  hence we have $\varphi_\lambda(r_1)=0=\varphi_\lambda(r_2)$, contradicting the
  assumption about the radii $r_1$, $r_2$.

  For the second assertion we work with the distributions
  $T_r\in\E'_0(X)$ given by $T_rf=\int_{B_r(x_0)}f$. As before,
  $T_r*f(x)=\int_{B_r(x)}f$ for all $x\in X$. Now, the proof proceeds as
  above, with the variety
  \begin{equation} \label{eq:V0}
  V = \left\{ f \in \E_0(X) \mid
  \int_{B_{r_1}(x)} f = 0 = \int_{B_{r_2}(x)} f \text{ for all } x \in X\right\}
  = \{ f\in\E_0(X) \mid T_{r_1}*f = 0 = T_{r_2}*f\}.
  \end{equation}
  Again, we conclude the existence of a $\varphi_\lambda$ satisfying
  $$ 0 = \int_{B_{r_i}(x)}\varphi_\lambda(x)dx =
        \omega_n \int_0^{r_i}\varphi_\lambda(\rho)\theta(\rho)d\rho= \omega_n \Phi_\lambda(r_i), $$
  where ${\rm dim}(X) = n+1$ and $\omega_n$ is the volume of the standard unit sphere of dimension $n$. As before, this contradicts to the choice of the $r_i$.
\end{proof}

Also harmonicity of a function follows from the mean value property
for two suitably chosen radii:

\begin{thm} \label{theomain3}
  Let $r_1, r_2 > 0$ be such that the equations
  \[ \varphi_\lambda(r_j)=1,\, j=1,2, \]
  have no common solution $\lambda \in \CC \backslash\{\pm iH/2\}$.

  Then $f \in C^\infty(X)$ is harmonic if and only if
  \[ \frac{1}{{\rm vol}(S_r(x))} \int_{S_r(x)} f = f(x)\]
  for $r=r_1,r_2$ and all $x \in X$.
\end{thm}

\begin{proof}
  We now use the distributions $T_r f = \frac{1}{\vol(S_r(x_0))}
  \left( \int_{S_r(x_0)} f \right) - f(x_0)$ and assume, as above, the
  existence of a function $g \in \E_0(X)$ with $T_{r_1} * g =
  T_{r_2} * g = 0$ and $\Delta g \neq 0$ (i.e., $g$ not
  harmonic). As in the proof of Theorem 1.3 of \cite{PS-10}, we
  consider the variety $V_0^g = \overline{\{ T * g \mid T \in \E_0'(X) \}}
  \subset \E_0(X)$, and show that the only non-zero functions
  $\varphi_{\lambda,k} \in V_0^g$ are $\varphi_{\pm i H/2} =
  1$. Therefore, $V_0^g$ consists only of constant functions,
  contradicting to $g \in V_0^g$ and $\Delta g \neq 0$.
\end{proof}

\section{Cheeger constant and heat kernel}\label{sec:cheeger}

The {\em Cheeger constant} $h(X)$ of a non-compact, complete
$n$-dimensional Riemannian manifold $(X,g)$ is defined as
\begin{equation} \label{eq:cheeg}
h(X) := \inf_{K \subset X} \frac{\vol(\partial K)}{\area(K)},
\end{equation}
where $K$ ranges over all connected, open
submanifolds of $X$ with compact closure and smooth boundary.
The {\em exponential volume growth} of $X$ is defined by
\begin{equation} \label{eq:mu}
\mu(X) := \limsup_{r \to \infty} \frac{\log \vol(B_r(x))}{r}.
\end{equation}
One easily checks that $\mu(X)$ does not depend on the choice $x
\in X$. The following result states that several fundamental constants of non-compact harmonic spaces agree.

\begin{thm} \label{thm:cheeg}
  Let $(X,g)$ be a non-compact, simply connected harmonic
  space and $H \ge 0$ be the mean curvature of its horospheres.
  Then we have the equalities
  $$ h(X) = H = \mu(X) = \lim_{r \to \infty} \frac{\log \vol(B_r(x)}{r}. $$
\end{thm}

\begin{proof}
  Our first goal is to prove $h(X) \ge H$. The proof is very similar
  to the proof of Theorem 3 in \cite{PS-03}. We refer the reader to
  this reference for more details. Let $\Psi: \RR \times \H_0 \to X$
  be the diffeomorphism introduced in Section
  \ref{sec:abelfourier}. We work in the space $X' = \RR \times \H_0$
  with the induced Riemannian metric $g' = \Psi^* g$. We know from
  Corollary \ref{cor:horvolform} that the volume element on $X'$ is
  given by $e^{sH} dt \wedge \omega_0$.

  W.l.o.g., we can assume $H > 0$, for otherwise there is nothing
  to prove. Let $P: X' \to \H_0$ be the canonical projection and $K
  \subset X'$ be an admissible set of \eqref{eq:cheeg}. Let $U$ be
  the projection of $K$ without the critical points of
  $P\vert_{\partial K}$. By Sard's theorem, $U$ has full measure in $P(K)$.
  For $x \in U$, let $f^{\pm}(x)$ be the maximum, resp., minimum of the set
  $\{ t \in \RR \mid (t,x) \in K \}$. Let $\widetilde K := \{ (x,t) \mid x \in U,
  f^-(x) \le t \le f^+(x) \}$. Then
  $$ \vol(K) \le \vol(\widetilde K) = \frac{1}{H} \int_U \left(
  e^{f^+(x) H} - e^{f^-(x) H} \right) \omega_0(x). $$
  Now we introduce the sets $\partial K^{\pm} := \{ (u,f^{\pm}(u))
  \mid u \in U \}$. Obviously, we have $\area(\partial K) \ge
  \area(\partial K^+) + \area(\partial K^-)$ and, analogously as
  in \cite{PS-03}, we obtain the estimate
  $$ \area(\partial K^{\pm}) \ge \int_U e^{f^\pm(x) H} \omega_0(x). $$
  This yields the desired estimate
  $$ \frac{\area(\partial K)}{\vol(K)} \ge \frac{\area(\partial K^+) +
  \area(\partial K^-))}{\vol(K)} \ge H. $$

Let $f(r) = \log \vol(B_r(x))$. Then, for all $r > 0$,
$$ f'(r) = \frac{\area(S_r(x))}{\vol(B_r(x))} \ge h(X). $$
It was shown in \cite{RaSha-03} that $A(r) := \area(S_r(x))$ is
strictly increasing in $r$ and that $\frac{A'}{A}$ is monotone
decreasing with limit $H \ge 0$. Applying l'Hospital's rule twice,
we conclude that
$$
\lim_{r \to \infty} \frac{f(r)}{r} = \lim_{r \to \infty} f'(r) =
\lim_{r \to \infty} \frac{A(r)}{\vol(B_r(x))} = \lim_{r \to
\infty} \frac{A'(r)}{A(r)} = H.
$$
Thus we have
$$ h(X) \le \mu(X) = H = \lim \frac{\log \vol(B_r(x)}{r}. $$
Both estimates together prove the theorem.
\end{proof}

\begin{rem}
  Ch. Connell proved in \cite{Co-00} that the Cheeger
  constant and the exponential volume growth of NCHS (simply connected
  strictly negatively curved homogeneous spaces) agree. This agreement
  fails when dropping the curvature condition: horospheres $\H$
  with barycentric normal directions in higher rank symmetric spaces
  of non-compact type are unimodular solvable groups with $h(\H) = 0$
  and $\mu(\H) > 0$ (see \cite{Pe-01}). The
  simplest example of this type is ${\rm Solv}(3)$, the diagonal horosphere
  in the product of two hyperbolic planes. Note that our Theorem
  \ref{thm:cheeg} does not contain any curvature condition.
\end{rem}

Applying Cheeger's inequality and Brooks' result $\lambda_0^{\rm ess}
\le \mu(X)^2/4$ (see \cite{Br-81}), we obtain

\begin{cor}
  Let $(X,g)$ be a non-compact, simply connected harmonic space and $H
  \ge 0$ be the mean curvature of its horospheres. Then the bottom of
  the spectrum and of the essential spectrum agree, and
  $$ \lambda_0(X) = \lambda_0^{\rm ess}(X) = \frac{H^2}{4}. $$
\end{cor}

Applying \cite[Prop. 2.4]{Kn-12}, we obtain

\begin{cor}
  Let $(X,g)$ be a non-compact, simply connected harmonic
  space of dimension $n$. If $X$ has vanishing Cheeger constant, then $X$ is
  isometric to the flat Euclidean space $\RR^n$.
\end{cor}

Finally, we consider the Abel transform of the heat kernel on
non-compact harmonic manifolds. We first state a useful lemma.

\begin{lem} \label{lem:expgrhoro}
  Let $\H \subset X$ be a horosphere and $x\in\H$. Then there is
  $C>0$ so that
  $$\vol_\H(B_r(x)\cap\H)\leq C e^{Cr} $$
  for all $r\geq 0$.
\end{lem}

\begin{proof}
  Without loss of generality, we can assume that $x=x_0$ and $\H_0 =
  \H$. We use the diffeomorphisms $\Psi_s: \H \to \H_s$, introduced
  earlier.

  Let $R > 0$ be fixed. Let $A_r=B_r(x_0)\cap\H$ and
  $$G_{R,r}=\bigcup_{s\in[-R,R]}\Psi_s(A_r)\ .$$
  We compute
  $$\vol(G_{R,r})=\int_{-R}^R\vol_{\H_s}(\Psi_s(A_r)) ds=\int_{-R}^Re^{sH}\vol_\H(A_r)ds=\frac{2\sinh(RH)}{H}\vol_\H(A_r)\ .$$
  By the triangle inequality, $G_{R,r}\subset B_{r+R}(x_0)$, and,
  since
  $$\vol(B_{r+R}(x_0))\leq C'e^{C'(r+R)}$$
  with some constant $C'>0$, by Bishop's volume comparison theorem, we
  have
  $$
  \vol_\H(A_r)=\frac{H\vol(G_{R,r})}{2\sinh(RH)}\leq\frac{HC'e^{C'(r+R)}}{2\sinh(RH)}=\frac{HC'e^{C'R}}{2\sinh(RH)}e^{C'r}.$$
\end{proof}

It is a well-known fact that a general complete Riemannian manifold
$(X,g)$ with Ricci curvature bounded from below has a unique heat
kernel $p_t^X(x,y)$ (see, e.g., \cite[Thm VIII.3]{Cha-84}). In the case
that $(X,g)$ is harmonic, the heat kernel is a {\em radial} kernel
function (see, e.g., \cite[Theorem 1.1]{Sz-90}), and is therefore
uniquely determined by the function $k_t^X(x) := p_t^X(x_0,x)$, where
$x_0 \in X$ is a fixed reference point. Our main result states that
the Abel transform of the heat kernel on a non-compact harmonic space
agrees, up to the factor $e^{-H^2t/4}$, with the Euclidean heat
kernel $k_t^\RR(s) = p_t^\RR(0,s) = \frac{1}{\sqrt{4 \pi t}}
e^{-s^2/(4t)}$. Since the heat kernel of a non-compact harmonic
manifold does not have compact support, one has to guarantee that its
Abel transform (centered at $x_0$), evaluated via the integral
\eqref{eq:Af} over the horospheres $\H_s = \Psi_s(\H_0)$, is
well-defined. This follows from the following result.

\begin{lem} \label{lem:hkernelest}
  Let $t > 0$ be fixed, $x_0 \in \H_0$ and $\Psi_s: \H_0 \to \H_s$ be the
  diffeomorphisms introduced earlier. Let $x_s = \Psi_s(x_0) \in \H_s$.
  For all $\epsilon > 0$, there exists an $r_0 > 0$ such that we have for all
  $s \in \RR$:
  \begin{equation} \label{eq:inthkernel}
  0 \le \int_{\H_s \backslash B_{2|s|+r_0}(x_s)} k_t^X(x)\, d\omega_s(x) \le
  \epsilon.
  \end{equation}
\end{lem}

\begin{proof}
  Since the Ricci curvature of the non-compact harmonic manifold
  $(X,g)$ is bounded below, there exist constants $C_t, \alpha_t > 0$
    such that
  \begin{equation} \label{eq:heatdecay}
    0 \le k_t^X(x) \le C_t e^{-\alpha_t r(x)^2} \qquad
    \text{for all $x \in X$},
  \end{equation}
  by a classical result of Li and Yau \cite{LY-86}. Using Lemma
  \ref{lem:expgrhoro}, we derive for arbitrary $r>0$:
  \begin{eqnarray*}
    \int_{\H_s \backslash B_r(x_s)} k_t^X(x)\, d\omega_s(x) &\le&
    \sum_{j=1}^\infty \int_{(\H_s \cap B_{(j+1)r}(x_s)) \backslash B_{jr}(x_s)}
    k_t^X(x)\, d\omega_s(x) \\
    &\le& \sum_{j=1}^\infty \vol(\H_s \cap B_{(j+1)r}(x_s)) C_t
    e^{-\alpha_t (jr-|s|)^2} \\
    &\le& C C_t \sum_{j=1}^\infty e^{C(j+1)r-\alpha_t (jr-|s|)^2}.
  \end{eqnarray*}
  Choosing $r = r_0 + 2|s|$ and $r_0 \ge \frac{2C}{\alpha_t} + \frac{1}{2}$,
  one easily sees that
  \begin{eqnarray*}
  C(j+1)r + \frac{\alpha_t}{2}j r_0 + 2 \alpha_t j r |s| &\le&  C(j+1)r + \frac{\alpha_t}{2}j r + 2 \alpha_t j r |s| \\
  &\le& \alpha_t j^2 r^2 \le \alpha_t j^2 r^2 + \alpha_t j^2 r,
  \end{eqnarray*}
  for all $j \ge 1$, which implies
  \begin{equation} \label{eq:ineq}
  C(j+1)r-\alpha_t (jr-|s|)^2 \le - \frac{\alpha_t}{2} j r_0.
  \end{equation}
  Choosing $r_0 \ge \frac{2C}{\alpha_t} +
  \frac{1}{2}$ even larger, if needed, we deduce from \eqref{eq:ineq}
  $$ \int_{\H_s \backslash B_r(x_s)} k_t^X(x)\, d\omega_s(x) \le C C_t
  \sum_{j=1}^\infty \left( e^{-\frac{\alpha_t}{2}r_0} \right)^j =
  C C_t \frac{e^{-\frac{\alpha_t}{2}r_0}}{1-e^{-\frac{\alpha_t}{2}r_0}} \le \epsilon, $$
  finishing the proof.
\end{proof}

\begin{thm} \label{thm:heatk}
  Let $(X,g)$ be a non-compact, simply connected harmonic
  space and $H \ge 0$ be the mean curvature of its horospheres. Then
  the Abel transform $\A k_t^X$ of the heat kernel $k_t^X(x) =
  p_t^X(x_0,x)$ is
  $$ (\A k_t^X)(s) = e^{-H^2t/4} \frac{1}{\sqrt{4\pi t}} e^{-s^2/4t}. $$
\end{thm}

\begin{proof} In the case $H=0$, $(X,g)$ is the Euclidean space and
  there is nothing to prove. So we can assume that $H > 0$.

  Since $\A k_t^X: \RR \to \RR$ is an even function, we have
$$ \int_{-\infty}^\infty \A k_t^X(s) ds = 2 \int_0^\infty e^{-\frac{H}{2}s}
\int_{\H_s} k_t^X(z) d\omega_s(z) \; ds \le 2 \int_0^\infty
\int_{\H_s} k_t^X(z) d\omega_s(z) \; ds \le 2 \int_X k_t^X(x) dx =
2, $$
by the heat conservation property $\int_X k_t^X(x) dx = 1$ for all $t > 0$
(see, e.g., \cite[Theorem 8.5]{Cha-84}). This shows that $\A k_t^X \in
L^1(\RR)$. Next, we show that $\A k_t^X$ is continuous. Let $s_0 \in
\RR$ and $\epsilon > 0$ be given. We conclude from Lemma
\ref{lem:hkernelest}, that there is an $r_0 > 0$ such that
$$ e^{-\frac{H}{2}s} \int_{\H_s \backslash B_{2s+r_0}(x_s)} k_t^X(x)d\omega_s(x)
\le \epsilon/3, $$
for all $s \in (s_0-1,s_0+1)$. Since the map
$$
s \mapsto F(s) := e^{-\frac{H}{2}s} \int_{\H_s \cap B_{r_0+2|s|}(x_s)}
k_t^X(x)d\omega_s(x)
$$
is obviously continuous, we can find $0 < \delta < 1$ such that
$$ | F(s) - F(s_0) | \le \epsilon/3, $$
for all $s \in (s_0-\delta,s_0+\delta)$. This implies that
\begin{multline*}
| \A k_t^X(s) - \A k_t^X(s_0) | \le \\
\left| e^{-\frac{H}{2}s} \int_{\H_s \backslash B_{2s+r_0}(x_s)} k_t^X(x)d\omega_s(x)\right| + \left| e^{-\frac{H}{2}s_0} \int_{\H_{s_0} \backslash B_{2s_0+r_0}(x_{s_0})} k_t^X(x)d\omega_{s_0}(x) \right|
+ | F(s) - F(s_0) | \le \epsilon,
\end{multline*}
for all $s \in (s_0-\delta,s_0+\delta)$. This shows that $\A k_t^X \in
C(\RR) \cap L^1(\RR)$.

For the proof of the theorem, it only remains to show that the Fourier
transforms of the $L^1$-functions $\A k_t^X$ and $e^{-H^2t/4} k_t^\RR$
agree. To show this, we need some growth information of $k_t^X$ and
$\varphi_\lambda$ and their derivatives.

Let us first consider $\varphi_\lambda$ for $\lambda \in \RR$. An
immediate consequence of $\varphi_\lambda = a \psi_\lambda$ is
$$ | \varphi_\lambda(r) | \le C e^{\frac{H}{2} r}, $$
with a suitable constant $C > 0$. Moreover, $\varphi'' +
\frac{\theta'}{\theta} \varphi' = L \varphi$ with $L = - \left(
  \lambda^2 + H^2/4 \right)$ implies that
$$ \Vert \nabla \varphi_\lambda(r) \Vert \le |L| \int_0^r
\underbrace{\left| \frac{\theta(t)}{\theta(r)} \right|}_{\le 1} |
\varphi_\lambda(t) | dt, $$
which shows that $\Vert \nabla \varphi_\lambda \Vert$ grows also at most
exponentially in the radius.

Next, we derive superexponential decay of the derivatives
$\frac{\partial}{\partial t} k_t^X$ and $\Vert \nabla k_t^X
\Vert$. Since the Ricci curvature of $(X,g)$ is bounded from below and
all balls of the same radius have the same volume, we conclude from
\cite[Prop. 1.1]{Gri-94} that
$$ p(x,x,t) \le C \begin{cases} e^{-\lambda_0(X) t},\quad
  \text{if $t \ge 1$}, \\
  t^{-n/2}, \quad \text{if $t \le 1$}, \end{cases} $$ with a suitable
constant $C > 0$. Since we have $\lambda_0(X) = H^2/4 > 0$, we can
find another constant $C' > 0$ such that
$$ p(x,x,t) \le \frac{C'}{t^{n/2}}, $$
for all $x \in X$ and $t > 0$. Then we are in Case 1 of \cite{Dav-89},
and Theorems 2 and 6 in \cite{Dav-89} imply that, for any fixed time
$t > 0$, the above heat kernel derivatives decay at the rate $C p(r)
e^{-r^2/4t}$, with a suitable polynomial $p$, i.e., the decay is
superexponential (i.e., the superexponential decay is quadratic in the
distance).

Let $\lambda \in \RR$ be fixed and
$$f(t) := \widehat{\A k_t^X}(\lambda) = \langle \A k_t^X, \psi_\lambda \rangle
= \langle k_t^X, a \psi_\lambda \rangle = \int_X
p_t^X(x_0,x)\varphi_\lambda(x) dx. $$

The last step of these identities follows easily from the exponential
decay \eqref{eq:heatdecay} of the heat kernel itself. We need to show that
$f(t)$ agrees with
$$g(t) := e^{-\frac{H^2}{4}t} \widehat{k_t^\RR}(\lambda) = e^{-\frac{H^2}{4}t}
e^{-\lambda^2 t}. $$
Obviously, both functions satisfy $\lim_{t \to 0} f(t) = \lim_{t \to 0} g(t) = 1$.
So it only remains to show that we have $f'(t) = g'(t)$ for all $t > 0$. Now,
\begin{eqnarray*}
  \frac{d}{dt} \int_X p_t^X(x_0,x)\varphi_\lambda(x) dx &=&
  \int_X \left( \frac{\partial}{\partial t} p_t^X(x_0,x) \right) \varphi_\lambda(x) dx \\
  &=& \int_X \left( \Delta_x p_t^X(x_0,x) \right) \varphi_\lambda(x) dx \\
  &=& - \int_X\left\langle\nabla k_t^X(x) , \nabla \varphi_\lambda(x)\right\rangle dx \\
  &=& \int_X p_t^X(x_0,x) \Delta \varphi_\lambda(x) dx\\
  &=& -\left( \lambda^2 + \frac{H^2}{4} \right) \varphi_\lambda(x_0) = g'(t).
\end{eqnarray*}
All steps in this calculation are justified by the growth properties derived
above. This finishes the proof.
\end{proof}

The corresponding result to Theorem \ref{thm:heatk} in the special
case of Damek-Ricci spaces can be found, e.g., in
\cite[(5.6)]{ADY-96}.

\bigskip

\noindent
Department of Mathematical Sciences, Durham University, Science Laboratories
South Road, Durham DH1 3LE, United Kingdom,
Email: norbert.peyerimhoff@durham.ac.uk

\smallskip

\noindent
Department of Mathematics and Statistics, University of
Cyprus, P.O. Box 20537, 1678 Nicosia, Cyprus,
Email:samiou@ucy.ac.cy

\end{document}